\documentclass[11pt,a4paper]{article}
\usepackage[]{geometry}
\usepackage[mathcal]{euscript}
\usepackage{bm,url}                     
\usepackage{amsfonts}
\usepackage{amssymb}
\usepackage{amsmath}
\usepackage{amsthm}
\usepackage[utf8x]{inputenc}
\usepackage{float}
\usepackage[pdftex]{graphicx}
\DeclareGraphicsExtensions{.bmp,.png,.pdf,.jpg,.ps}
\usepackage{colortbl}  
\usepackage{todonotes}             
\usepackage{booktabs}
\newtheorem{definition}{Definition}
\newtheorem{remark}{Remark}
\newtheorem{theorem}{Theorem}
\newtheorem{lemma}{Lemma}
\newtheorem{proposition}{Proposition}

\newtheorem{notation}[theorem]{Notation}

\usepackage[english]{babel}
\usepackage{amsmath}
\usepackage{amsfonts}
\usepackage{amssymb}
\usepackage{graphicx}
\usepackage{color}
\usepackage{array}
\usepackage{mathrsfs}
\usepackage{hyperref}

\definecolor{violet}{rgb}{0.7,0,0.6}
\definecolor{OliveGreen}{RGB}{85,107,47}

\title{Parameter estimation for the FOU(p) process with
the same lambda}
\author{Juan Kalemkerian\\
Universidad de la República, Facultad de Ciencias.}

\begin{document}
\maketitle
\begin{abstract}
\noindent The FOU$(p)$ processes can be considered as an
alternative to ARMA (or ARFIMA) processes to model time series. Also, 
there is no substantial loss when we model a time series using FOU$(p)$ processes
with the same $\lambda$, than using differents  $\lambda$'s. 
In this work we propose a new method to estimate the unique value of $\lambda$ in a FOU$(p)$ process.
Under certain conditions, we will prove consistency and asymptotic normality.
We will show that this new method is more easy and fast  to compute. 
 By simulations, we show that 
the new procedure  work well and is more efficient than the general method. Also, we include an application to real 
data, 
and we show that the new method work well too and outperforms the family of ARMA$(p,q)$.

\end{abstract}

AMS: 62M10

\newpage

\section{Introduction}\label{introduction}
In \cite{chichi} the FOU$(p)$ processes was introduced. The FOU$(p)$ processes are a continuous time 
centered and stationary Gaussian process, and is  obtained
from the iteration of the  $T_\lambda$ functional $p$ times (in
Section \ref{preliminaries} can be viewed the explicit definition of this functional) applied to a fractional Brownian motion 
(fBm), where $\lambda$
is a positive number. The FOU$(p)$ processes contains several parameters, for one side $\sigma$ and $H$ are 
the parameters of the fBm (scale and Hurst parameters respectively) and for  other side
$\lambda_1, \lambda_2,...,\lambda_q$ are
parameters from the application of the functional $T_{\lambda}$ $p$ times, in such way that
 applying the $T_{\lambda_i}$ functional $p_i$ times for $i=1,2,...,q$,  where $p_1+p_2+...+p_q=p$ 
 ($p$ is the total number of iterations).
More explicitely, the authors proposed to use the notation
 FOU$ \left( \lambda _{1}^{\left(
		p_{1}\right) },\lambda _{2}^{\left( p_{2}\right) },...,\lambda _{q}^{\left(
		p_{q}\right) },\sigma ,H\right).$ 
		In the particular case in which $q=1$ we have a FOU$(\lambda^{(p)},\sigma,H)$ process.
The FOU$(p)$ process has several interesting theoretical properties. First, the fractional Ornstein--Uhlenbeck 
process
defined by \cite{Cheridito} is a particular case of FOU$(p)$ by taking $p=1$. Second, when $H>1/2$
 any FOU$(p)$
has short range dependence for $p \geq 2$ and long range dependence for $p=1$. Thirth, as $p$ grows,   
the autocorrelation function of the process goes more quickly to zero.
Fourth, 
any FOU$(p)$ has an explicit and simple formula for the spectral density. 
Fiveth, using a result given in \cite{Ibragimov}, the parameter $H$ is the local
Hölder index of the process.
Sixth,
if we have observed the process
in an equispaced sample of $[0,T]$ where $T,n \rightarrow + \infty$, under certain conditions 
between $T$ and $n$, it is possible to 
estimate consistently all the parameters of any FOU$(p)$. 
The parameters $H$ and $\sigma$ can be estimated by a procedure proposed in \cite{Istas}, and
$\hat{H}$ and $\hat{\sigma}$ have an explict formula from the observed data. As a second step, from the formula
for the spectral density, it is possible to estimate the $\lambda_i$ parameters by 
 using a modified Whittle contrast. To obtain
these estimators it is necessary to optimize a function that not have an explicit 
formula, thus the optimum must be found by a numerical optimization procedure.
The proof of the above mentioned properties can be found in \cite{chichi} and
\cite{kalemkefou}. Also, the family of FOU$(p)$ can be used to model a wide range of time series, including
short memory and long memory time series. In \cite{kalemkefou} can be found three examples of 
real data modeled by FOU$(p)$ and your comparison with ARMA, or ARFIMA models, and can be seen 
the good performance (including measures proposed in \cite{Wilmott}) of this new class of models. 
In this work, in Section \ref{application to real data}, 
we will add a fourth example of real data
modeled better than the family of ARMA models.
In Section \ref{preliminaries} we give a basis  to understand the definition of the FOU$(p)$ processes.
In Section \ref{parameter estimation} we describe the procedure of the parameter estimation of any FOU$(p)$ model
proposed in \cite{chichi} and \cite{kalemkefou} (subsections 3.1 and 3.2) and we propose in Subsection 3.3
the main theoretical result of this work, that is to estimate from an explicit formula the parameter
$\lambda$ when we have a FOU$(\lambda^{(p)},\sigma,H)$ process, and to prove the consistency and asymptotic 
normality of $\hat{\lambda}$.
In Section \ref{a comparison} we corroborate the theoretical results (consistency and asymptotic normality) by simulations 
and we show a comparison between the formula to estimate $\lambda$ given in Subsection 3.3 with the proposed
in \cite{chichi} and we show an improvement in terms of slightly diminution of the variance of the estimator.
In Section \ref{application to real data} we show an application to  real data set of the proposed estimation procedure to fit 
a FOU$(p)$ model and we show
an improving results than the ARMA family.
Our conclusions are given in Section \ref{conclusions} and the proof of the theoretical results
proposed in Subsection 3.3 are given in Section \ref{proofs}.

\section{Preliminaries}\label{preliminaries}
We start recalling the definition of a fractional Brownian motion.

\begin{definition}
	A fractional Brownian motion with Hurst parameter $H\in \left( 0,1\right] $,
	is an almost surely continuous centered Gaussian process $\left\{
	B_{H}(t)\right\} _{t\in \mathbb{R}}$ with 
	\begin{equation*}
	\mathbb{E}\left( B_{H}(t)B_{H}(s)\right) =\frac{1}{2}\left( \left\vert
	t\right\vert ^{2H}+\left\vert s\right\vert ^{2H}-\left\vert t-s\right\vert
	^{2H}\right) ,\text{ \ }t,s\in \mathbb{R}.
	\end{equation*}
\end{definition}

We follow with the definition of the iterated Ornstein-Uhlenbeck processes
of order $p$ defined in \cite{chichi} (FOU$(p)$). 

\begin{definition}\label{definition_FOU}
	Suppose that $\left\{ \sigma B_{H}(s)\right\} _{s\in \mathbb{R}}$ is a fractional
	Brownian motion with Hurst parameter $H,$ and scale parameter $\sigma$.
	Suppose further that $\lambda _{1},\lambda
	_{2},...,\lambda _{q}$ are distinct positive numbers and that $p_{1},p_{2},...,p_{q}\in \mathbb{N}$ are such that $%
	p_{1}+p_{2}+...+p_{q}=p$. We define the iterated Ornstein-Uhlenbeck process
	of order $p$ as $\left\{ X_{t}\right\} _{t\in \mathbb{R}%
	} $ by 
	\begin{equation*}
	X_{t}:= T_{\lambda _{1}}^{p_{1}} \circ T_{\lambda _{2}}^{p_{2}} \circ ....
	\circ T_{\lambda _{q}}^{p_{q}} (\sigma
	B_{H})(t)=\sum_{i=1}^{q}K_{i}\left( \lambda \right)
	\sum_{j=0}^{p_{i}-1}\binom{p_{i}-1}{j} T_{\lambda _{i}}^{\left(
		j\right) }(\sigma B_{H})(t),
	\end{equation*}%
	where the numbers $K_{i}\left( \lambda \right) $  are defined by \begin{equation}
K_{i}\left( \lambda \right) =K_{i}\left( \lambda _{1},\lambda
_{2},...,\lambda _{q}\right): =\frac{1}{\prod\limits_{j\neq i}\left(
	1-\lambda _{j}/\lambda _{i}\right) }  \label{k_i}
\end{equation} and the operators $%
	T_{\lambda _{i}}^{\left( j\right) }$  follows the next formula 
	\begin{equation}
T_{\lambda }^{(h)}(y)(t):=\int_{-\infty }^{t}e^{-\lambda (t-s)}\frac{\left(
	-\lambda \left( t-s\right) \right) ^{h}}{h!}dy(s) \ \ \text {for}  \ \ h=0,1,2,... \label{hh}
\end{equation}
When $h=0$ we simply call $T_\lambda$, thus \begin{equation}
T_{\lambda }(y)(t):=\int_{-\infty }^{t}e^{-\lambda (t-s)} dy(s). \label{h=0}
\end{equation}
\label{fou_definition}
\end{definition}

\begin{notation}
	$\left\{ X_{t}\right\} _{t\in \mathbb{R}}\sim  \text{FOU } \left( \lambda _{1}^{\left(
		p_{1}\right) },\lambda _{2}^{\left( p_{2}\right) },...,\lambda _{q}^{\left(
		p_{q}\right) },\sigma ,H\right), $ where $0<\lambda_1<\lambda_2<...<\lambda_q$ or more simply,
	$\left\{ X_{t}\right\} _{t\in \mathbb{R}}\sim$FOU$(p)$.
	
\end{notation}

\begin{remark}
 In \cite{Arratia} can be found the proof of the equality $T_{\lambda _{1}}^{p_{1}} \circ T_{\lambda _{2}}^{p_{2}} \circ ....
	\circ T_{\lambda _{q}}^{p_{q}} (\sigma
	B_{H})(t)=\sum_{i=1}^{q}K_{i}\left( \lambda \right)
	\sum_{j=0}^{p_{i}-1}\binom{p_{i}-1}{j} T_{\lambda _{i}}^{\left(
		j\right) }(\sigma B_{H})(t)$ given in Definition \ref{definition_FOU}.
\end{remark}

\begin{remark}
	When $p=1$, we obtain a fractional Ornstein--Uhlenbeck process (FOU$\left(
	\lambda ,\sigma ,H\right) $).
\end{remark}

Throughout this work we will consider the case $q=1$, this is $\left\{ X_{t}\right\} _{t\in \mathbb{R}}\sim  \text{FOU } 
\left( \lambda ^{\left(
		p\right) },\sigma ,H\right). $
\begin{remark}
As suggested in \cite{kalemkefou}, to model a time series data set from a FOU$(p)$ process,
in several cases may be convenient to standardize the data and then fit  FOU$(p)$ 
where $\sigma=1$. In this way we avoid the estimation of $\sigma$. In this work we will use the notation
 FOU$(\lambda^{(p)}, H)$ for any FOU$(\lambda^{(p)},\sigma, H)$ where $\sigma=1$.
\end{remark}
\begin{remark}
 It is possible to apply the Definition 2, evaluating the functional $T_\lambda$ at another process 
instead $\sigma B_H $. Some general properties of this class of processes can be found in \cite{Arratia}.
\end{remark}

\section{Parameter estimation}\label{parameter estimation}
In \cite{chichi} it is proposed a method to estimate all the parameters of any FOU$(p)$ in a consistent way and
the estimators of $\sigma$ and $H$ have asymptotic Gaussian distribution if the process is observed in a equispaced
sample of $[0,T]$. In addittion, if the process is observed throughout the interval $[0,T]$, the $\lambda's$ parameters
also have asymptotic Gaussian distribution. In \cite{kalemkefou} is proposed a consistent way to estimate the $\lambda's$ 
parameters when the process is observed in an equispaced sample of $[0,T]$.
In the FOU$\left( \lambda ^{\left(
p\right) },\sigma ,H\right)$ case, we propose to estimate $(\sigma,H)$ in the same way as \cite{chichi}, but 
in the case in which $1/2<H<3/4$, we propose a plug-in formula to estimate $\lambda$ and we will show that this
estimator is consistent and has asymptotic Gaussian distribution. Also, we will show by simulations that this estimator
has less variance that the one proposed in \cite{kalemkefou}.

\subsection{Estimation of $H$ and $\sigma$ }

We start defining a filter of length $k+1$ an order $L$.
\begin{definition}
 $a=\left( a_{0},a_{1},...,a_{k}\right) $\ is  a \textit{filter of length $k+1$\ and order $L\geq 1$}\ if and only if the following 
conditions hold:

\begin{itemize}
	\item $\sum_{i=0}^{k}a_{i}i^{l}=0$\ para todo $0\leq l\leq L-1.$
	
	\item $\sum_{i=0}^{k}a_{i}i^{L}\neq 0.$
	
\end{itemize}
\end{definition}

Observe that given $a$ a filter of order $L$\ and length $k+1$, the new filter 
$a^{2}=\left( a_{0},0,a_{1},0,a_{2},0,...0,a_{k}\right) $ has order $L$\ and length $2k+1$. 
Now, we define the quadratic variation of a sample associated to a filter $a$ as follows.
In this work we will use the filters 
\begin{equation}                  
      a_k= \left ( -1, \binom{k}{1}, -\binom{k}{2},...,(-1)^{k-1}\binom{k}{2}, (-1)^{k}\binom{k}{1}, 
      (-1)^{k+1}\right ).
      \end{equation}
It is easy to see that $a_k$ is a filter of order $k$ and length $k+1$.

\begin{definition}
	Given a filter $a$\ of length $k+1$ and a sample $X_{1},X_{2},...,X_{n}$, we define 
	the quadratic variations associated with filter $a$ by
	\begin{equation*}
	V_{n,a}:=\frac{1}{n}\sum_{i=0}^{n-k}\left( \sum_{j=0}^{k}a_{j}X_{i+j}\right)
	^{2}.
	\end{equation*}
\end{definition}

The following theorem define $\left ( \hat{H}, \hat{\sigma} \right )$ and summarizes their asymptotic properties. 

\begin{theorem}[Kalemkerian \& Le\'on]\label{asymptotic of H sigma}\[\]
	If $X_{\Delta },X_{2\Delta },....,X_{i\Delta },...,X_{n\Delta }=X_{T}$ is an equispaced sample of the process $\left\{ X_{t}\right\} _{t\in 
		\mathbb{R}}\sim $FOU$\left(p\right) $ where $H>1/2$, the filter $a$\ is of order $L\geq 2$ and length $k+1$, $\Delta_n=n^{-\alpha}$ for some 
		$\alpha$ such that $0<\alpha<\frac{1}{2(2H-1)}$ and $T=n\Delta_n \rightarrow +\infty$,
	as $n\rightarrow +\infty. $ Define 
	\begin{equation}
\widehat{H}=\frac{1}{2}\log _{2}\left( \frac{V_{n,a^{2}}}{V_{n,a}}\right), 
\label{Hgorro}
\end{equation}
\begin{equation}
\widehat{\sigma }=\left( \frac{-2V_{n,a}}{\Delta _{n}^{2\widehat{H}%
	}\sum_{i=0}^{k}\sum_{j=0}^{k}a_{i}a_{j}\left\vert i-j\right\vert ^{2\widehat{%
			H}}}\right) ^{1/2}. \label{sigmagorro}
\end{equation}
Then

	\begin{enumerate}
		\item 
		\begin{equation*}
		\left( \widehat{H},\widehat{\sigma }\right) \overset{a.s.}{\rightarrow }%
		\left( H,\sigma \right) .
		\end{equation*}
		
		\item 
		\begin{equation*}
		\sqrt{n}\left( \widehat{H}-H\right) \overset{w}{\rightarrow }N\left(
		0,\Gamma _{1}\left( H,\sigma ,a\right) \right)
		\end{equation*}
		
		\item 
		\begin{equation*}
		\frac{\sqrt{n}}{\log n}\left( \widehat{\sigma }-\sigma \right) \overset{w}{%
			\rightarrow }N\left( 0,\Gamma _{2}\left( H,\sigma ,a\right) \right)
		\end{equation*}
	\end{enumerate}
\end{theorem}

\begin{remark}
In \cite{kalemkefou} it is showed that in the case in which $H<1/2$, the theorem remains valid  taking $\alpha >1/2$.
\end{remark}

\subsection{Estimation of  $\lambda$}
In \cite{chichi} it is obtained an explicit formula for the spectral density of any FOU$(p)$ and it is proposed
a modiffied Whittle procedure to estimate $\lambda = (\lambda_1, \lambda_2,....,\lambda_q)$ when the process is observed
in the whole interval $[0,T]$ and $H, \sigma$ are known. More explicitely, $\widehat{\lambda 
	}_{T}=\arg \min_{\lambda \in \Lambda }U_{T}\left( \lambda \right) $ where 
	$U_{T}\left( \lambda \right) =\frac{1}{4\pi }\int_{-\infty }^{+\infty }\left(
	\log f^{\left( X\right) }\left( x,\lambda \right) +\frac{I_{T}\left(
		x\right) }{f^{\left( X\right) }\left( x,\lambda \right) }\right) w\left(
	x\right) dx$ being $I_{T}\left( x\right) $ is the
	periodogram of the second order 
	$
	I_{T}\left( x\right) =\frac{1}{2\pi T}\left\vert
	\int_{0}^{T}X_{t}e^{-itx}dt\right\vert ^{2}$, 
	$f^{(X)}(x)=\frac{\sigma^2\Gamma(2H+1)\sin(H\pi)|x|^{2p-1-2H}}{2\pi\prod_{i=1}^q(\lambda^2_
	i+x^2)^{p_i}}$ is the spectral density of the process,
	$\Lambda$ is a compact set and $w$ be certain weight function.
	From this procedure and using a result obtained in \cite{Leo}, in \cite{chichi}
	can be see the proof of the consistency and asymptotic normality of $\widehat{\lambda 
	}_{T}$. In \cite{kalemkefou} it is show that it is possible to take a discretized version of $U_T$ and 
	$I_T$ and using $(\hat{H},\hat{\sigma})$ instead the true value of $(H,\sigma)$, then
	under certain condition
	of the weight function $w$ an the speed in which $T/n \rightarrow 0$, we have that $\widehat{\lambda 
	}_{T}$ is consistent.
	 Of course, this method has the drawback in terms of computational cost, due to the lack of explicit 
	 formula for the function $U_T$ and due to the fact that the algorithms to optimize functions beginning
	 with a initial values, and the final estimation can depend of its.
\subsection{An alternative procedure to estimate $\lambda$}	
In this subsection, we propose to estimate $\lambda$ from an explicit formula where $X_1,X_2,...,X_n$ be an 
equispaced sample of  $%
\left[ 0,T\right] $ of $\left\{ X_{t}\right\} _{t\in \mathbb{R}}\sim $FOU$%
\left( \lambda ^{\left( p\right) },\sigma ,H\right) $ and we will prove consistency and asymptotic
normality where $1/2<H<3/4$. Two advantages has this procedure with respect
to the procedure tu estimate $\lambda$ given in the previous subsection. First is that this procedure is very fast 
to calculate, and second we obtain consistency and asymptotic normality of the estimator.

\noindent  In the following proposition, we will show that when  $\left\{ X_{t}\right\} _{t\in \mathbb{R}}\sim  \text{FOU } 
\left( \lambda ^{\left(
p\right) },\sigma ,H\right) $ the variance of any observation $X_t$ verify a simple and explicit formula.

\begin{proposition}\[\]
 If $\left\{ X_{t}\right\} _{t\in \mathbb{R}}\sim $FOU$\left( \lambda
^{\left( p\right) },\sigma ,H\right) $ then 
\begin{equation}
\mathbb{V}\left( X_{t}\right) =\frac{\sigma ^{2}H\Gamma \left( 2H\right)
\prod_{i=1}^{p-1}\left( i-H\right) }{\left( p-1\right) !\lambda ^{2H}},
\label{V(Xt)}
\end{equation}
where $\prod_{i=1}^{p-1}\left( i-H\right)$ is defined as $1$ when $p=1$.
\label{fouvariance}
\end{proposition}
From this simple formula and knowing estimate $\mathbb{V}(X_t), \sigma$ and $H$,
we can obtain an explicit formula to estimate $\lambda$.

If we call $\widehat{\mu }_{2}=\frac{1}{n}\sum_{i=1}^{n}X_{i}^{2}$, put $%
\widehat{\mu }_{2}$ instead $\mathbb{V}\left( X_{t}\right) $ and changing $%
\left( \widehat{\sigma },\widehat{H}\right) $ instead $\left( \sigma
,H\right) $ we can obtain a natural plug-in estimator of $\lambda $ from 
\begin{equation}
\widehat{\lambda }=\left( \frac{\widehat{\sigma }^{2}\widehat{H}\Gamma
\left( 2\widehat{H}\right) \prod_{i=1}^{p-1}\left( i-\widehat{H}\right) }{%
\left( p-1\right) !\widehat{\mu }_{2}}\right) ^{\frac{1}{2\widehat{H}}}.
\label{lambdahat}
\end{equation}%
Having good asymptotic properties of $\left( \widehat{\sigma },\widehat{H}%
\right) $, from (\ref{lambdahat}), it is natural to find similar properties
for $\widehat{\lambda }.$ In the next theorem, we will show that when $1/2<H<3/4$ and adding 
a hypothesis about the speed in which $T$ goes to infinite, we can obtain consistency 
and asymptotic normality of $\hat{\lambda}$.
 The equality \ref{lambdahat}, is a generalization of the formula proposed in \cite{Brouste} to estimate
 the $\lambda$ parameter in a fractional Ornstein--Uhlenbeck process.
\begin{theorem}\[\]
 If $X_{1},X_{2},...,X_{n}$ be an equispaced sample observed in $%
\left[ 0,T\right] $ of $\left\{ X_{t}\right\} _{t\in \mathbb{R}}\sim $FOU$%
\left( \lambda ^{\left( p\right) },\sigma ,H\right) $, $1/2<H<3/4,$ $n\left( 
\frac{T}{n}\right) ^{k}\rightarrow 0,$ $\frac{T\log ^{2}n}{n}\rightarrow 0$
as $n\rightarrow +\infty $ for some $k>1,$ $T\rightarrow +\infty $, then 
\begin{equation*}
\widehat{\lambda }\overset{a.s.}{\rightarrow }\lambda 
\end{equation*}%
and 
\begin{equation*}
\sqrt{T}\left( \widehat{\lambda }-\lambda \right) \overset{w}{\rightarrow }%
N\left( 0,\alpha \left( \sigma ,H,\lambda \right) \right) .
\end{equation*}
\label{consistency}
\end{theorem}
\begin{remark}
 Observe that $T_n=\log n$ or $T_n=n^{\alpha}$ where $0<\alpha <1/2$
 (by taking $k=2$) verify the conditions requested 
 by Theorem \ref{consistency}.
\end{remark}

\section{A comparison}\label{a comparison}

In this section we will show a comparison between the performance of the estimator $\hat{\lambda}$ 
proposed in this work that we will call $\hat{\lambda}_p$ (plug-in) with the estimator 
proposed in \cite{chichi} that we will call $\hat{\lambda}_U$ (minimizing the function $U=U_T$).
In addittion, from the same simulation study,  we corroborate the consistency and asymptotic normality
of $\hat{\lambda}_p$. 
Althought we didn't have a theoretical result, we include in Table 1 and Table 4, the results for $H=0.3$.
Tables 1 to 3, show the  mean estimation 
$\left ( \hat{\lambda} \right )$, mean error estimation 
$\left (\left |\hat{\lambda}-\lambda^{0}\right |\right ) $ being $\lambda^{0}$  the true value of
$\lambda$ and deviation of the estimator
$\left ( \text{sd}\left (\hat{\lambda}\right )\right )$ for $m=100$  replications, when the 
observed process
is a FOU$(\lambda^{(2)},H)$  where $\lambda=0.8$ and different
values of $H$, viewed in $n$ equispacied points of $[0,T]$, 
for different values of $T$ and $n$.
Tables 1 to 3, shows that (for both estimators) the mean error estimation is not necessarily decreasing as $n$
increasing, showing that it is very important the relation between $T$ and $n$. Anyway, in all the cases
considered, the mean error estimation take small values. The same occurs with the deviation of the estimators.
Table 3 show that when $H=0.7$, in all the considered cases, we have that 
sd$\left (\hat{\lambda}_p \right ) <$ sd$\left (\hat{\lambda}_U \right )$. The same occurs for $H=0.5$
and $H=0.3$ for values of $T=100$ and $T=50$.
In almost all the cases for $H=0.5$ and $H=0.7$, the mean error estimation is less for 
 $\left (\hat{\lambda}_p \right )$ than for  $\left (\hat{\lambda}_U \right )$.
 In general, tables 1 to 3, shows  better results for $\hat{\lambda}_p$ than for 
 $\hat{\lambda}_U$.
 Table 4 show the p-value of the Truncated Cram\'er-von Mises test of normality for $\hat{\lambda}_p $ proposed in 
 \cite{tcvm}.
For $H=0.5$ and $H=0.7$,  we non reject normality for all the values of $T$ and $n$ considered (according
with our theoretical results).
In the $H=0.3$ case we reject normality only for $T=25$, but for $T=50$ and $T=100$ the test non reject normality.
Tables 1 to 4 suggest that for values of $H<1/2$ it is possible to have consistency and asymptotic normality 
when $n,T$ goes to inifnite where $T/n \rightarrow 0$ with certain velocity, this remains as an open problem under which 
conditions this assertion hold. To estimate $H$, we have used the  Daubechies' filter of order $2$, $a=$\\
 $ \frac{1}{\sqrt2}(.4829629131445341,-.8365163037378077,.2241438680420134,.1294095225512603)$. 
 
The results for other filters were similar.
\begin{table*}[ht]\caption{ Comparison between $\hat{\lambda}_U$ and $\hat{\lambda}_p$ as
estimators of $\lambda$. We report the mean estimation 
$\left ( \hat{\lambda} \right )$, mean error estimation 
$\left (\left |\hat{\lambda}-\lambda^{0}\right |\right )$ and deviation 
$\left ( sd\left (\hat{\lambda}\right )\right )$ 
for a FOU$(\lambda^{(2)},H)$ viewed in $n$ equispacied points of $[0,T]$, 
where $\lambda=0.8$ and $H=0.3$ for $m=100$  replications. }
 \label{tH03ll}

\centering
\begin{tabular}{|c|c|c|c||c|c||c|c|}

    \hline
   $T$ & $n$ & $\hat{\lambda}_U$    & $\hat{\lambda}_p$ & $\left | \hat{\lambda}_U  - \lambda^{0}\right | $ & $\left | \hat{\lambda}_p  - \lambda^{0}\right | $
 & sd$\left ( \hat{\lambda}_U \right )$& sd $\left (\hat{\lambda}_p \right )$  \\
   \hline 
   $100$ & $1000$  &0.7536 &0.7416 &0.0464 &0.0584 &0.2185 &0.2101 \\
           &  $5000$  &0.7955 &0.7874 &0.0045 &0.0126 &0.1671 & 0.1661 \\
           & $10000$  &0.8265 &0.8089 &0.0265 &0.0089 &0.1462 &0.1244\\
            \hline
     $50$ & $1000$  &0.7713 &0.7940 &0.0287 &0.0060 &0.2710 &0.2595 \\
            & $5000$   &0.8249 &0.7708 &0.0249 &0.0292 &0.2380 &0.2281 \\
            & $10000$  &0.8199 &0.8022 &0.0199 &0.0022 &0.1983 &0.2073\\

   \hline
     $25$ & $1000$   &0.7205 &0.8258 &0.0795 &0.0258 &0.2550 &0.3627 \\
            & $5000$    &0.8605 &0.8604 & 0.0605 &0.0604 &0.2842 &0.3284 \\
            & $10000$  &0.8742 &0.8142 &0.0742 &0.0142 &0.2610 &0.2536 \\

   \hline

\end{tabular}
\end{table*}

\begin{table*}[ht]\caption{ Comparison between $\hat{\lambda}_U$ and $\hat{\lambda}_p$ as
estimators of $\lambda$. We report the mean estimation 
$\left ( \hat{\lambda} \right )$, mean error estimation 
$\left (\left |\hat{\lambda}-\lambda^{0}\right |\right )$ and deviation 
$\left ( sd\left (\hat{\lambda}\right )\right )$ 
for a FOU$(\lambda^{(2)},H)$ viewed in $n$ equispaced points of $[0,T]$, 
where $\lambda=0.8$ and $H=0.5$ for $m=100$  replications. }
 \label{tH05ll}

\centering
\begin{tabular}{|c|c|c|c||c|c||c|c|}

    \hline
   $T$ & $n$  & $\hat{\lambda}_U$    & $\hat{\lambda}_p$ & $\left | \hat{\lambda}_U  - \lambda^{0}\right | $ & $\left | \hat{\lambda}_p  - \lambda^{0}\right | $
 & sd$\left ( \hat{\lambda}_U \right )$& sd $\left (\hat{\lambda}_p \right )$  \\
   \hline 
   $100$ & $1000$  &0.7514 &0.7536 &0.0486 &0.0464 &0.1980 &0.1400 \\
           &  $5000$  &0.7969 &0.8021 &0.0481 &0.0021 &0.1840 &0.1330 \\
           & $10000$  &0.8159 &0.8126 &0.0159 &0.0126 &0.1622 &0.1228\\
            \hline
     $50$ & $1000$  &0.7673 &0.7880 &0.0337 &0.0120 &0.2633 &0.1989 \\
            & $5000$   &0.8358 &0.8309 &0.0358 &0.0309 &0.2132 &0.1743 \\
            & $10000$  &0.8135 &0.7879 &0.0135 &0.0121 &0.1968 &0.1666\\

   \hline
     $25$ & $1000$   &0.7331 &0.7977 &0.0669 & 0.0023 &0.2151 &0.2282 \\
            & $5000$    &0.8541 &0.8178 &0.0541 &0.0178 &0.2312 &0.2235 \\
            & $10000$  &0.8153 &0.7963 &0.0153 &0.0037 &0.2850 &0.2030 \\

   \hline

\end{tabular}
\end{table*}

\begin{table*}[ht]\caption{ Comparison between $\hat{\lambda}_U$ and $\hat{\lambda}_p$ as
estimators of $\lambda$. We report the mean estimation 
$\left ( \hat{\lambda} \right )$, mean error estimation 
$\left (\left |\hat{\lambda}-\lambda^{0}\right |\right )$ and deviation 
$\left ( sd\left (\hat{\lambda}\right )\right )$ 
for a FOU$(\lambda^{(2)},H)$ viewed in $n$ equispacied points of $[0,T]$, 
where $\lambda=0.8$ and $H=0.7$ for $m=100$  replications. }
 \label{tH07ll}

\centering
\begin{tabular}{|c|c|c|c||c|c||c|c|}

    \hline
   $T$ & $n$ & $\hat{\lambda}_U$    & $\hat{\lambda}_p$ & $\left | \hat{\lambda}_U  - \lambda^{0}\right | $ & $\left | \hat{\lambda}_p  - \lambda^{0}\right | $
 & sd$\left ( \hat{\lambda}_U \right )$& sd $\left (\hat{\lambda}_p \right )$  \\
   \hline 
   $100$ & $1000$  &0.7209&0.7353 &0.0791  &0.0647 & 0.1890 &0.1325 \\
           &  $5000$  &0.7999&0.8104 &0.0001 &0.0136 &0.1594 &0.1114 \\
           & $10000$  &0.8067&0.8012 &0.0067 &0.0012 &0.1411 &0.0932\\
            \hline
     $50$ & $1000$  &0.8042&0.7703 &0.0042 &0.0297 &0.2245  &0.1815 \\
            & $5000$   & 0.8379&0.8125 &0.0379 &0.0125  &0.1868  &0.1467  \\
            & $10000$  &0.7929&0.7931  &0.0071 &0.0069 &0.1993 &0.1314\\

   \hline
     $25$ & $1000$   & 0.8322&0.8106 &0.0322 &0.0106 &0.3440&0.2168 \\
            & $5000$    & 0.8490&0.8322 &0.0490 &0.0322 &0.2432 &0.2030 \\
            & $10000$  & 0.8410&0.8298 &0.0410 &0.0298 &0.2319 &0.1908\\

   \hline

\end{tabular}
\end{table*}

   \begin{table*}[ht]\caption{ p-value for the Truncated Cram\'er-von Mises test of normality
    for $\hat{\lambda}$ for the FOU$(\lambda^{(2)},H)$ process viewed in $n$ equispaced points of $[0,T]$, 
where $ \lambda=0.8$, in cases $H=0.3$, $H= 0.5$ and $H= 0.7$  for $m=100$  replications. }
 \label{pvalue_lambda}

 \centering
\begin{tabular}{|c|c||c|c|c|}

    \hline
   $T$ & $n$\ \ \ & $H=0.3$ & $H=0.5$ & $H=0.7$    \\
   \hline 
  $25$ & $1000$  &  0.001 & 0.177 & 0.942   \\
           &  $5000$ &  0.022 & 0.502 & 0.289 \\
           & $10000$ & 0.019 & 0.268 & 0.160  \\
           
           \hline
       $50$ & $1000$  &   0.928 &0.508 &0.239   \\
           &  $5000$ &  0.236 & 0.229 & 0.252 \\
           & $10000$ &  0.354 & 0.358 & 0.490  \\
           
           \hline
  
     $100$ & $1000$  &  0.872& 0.437 &0727    \\
           &  $5000$ &  0.704 &0.198 &0.201  \\
           & $10000$ & 0.091 & 0.848 & 0.491  \\
           
           \hline
     
\end{tabular}
\end{table*}

\section{Application to real data}\label{application to real data}
The oxygen saturation in blood of a newborn child has been monitored during seventeen hours. 
We have observed $304$ measures taken at intervals of $200$ seconds ($X_1,X_2,...,X_{304}$).
We have standardized the data set and 
fitted FOU$(\lambda^{(p)},\sigma, H)$ and FOU$(\lambda^{(p)}, H)$ for $p=1,2,3,4$ and compared
the performance with ARMA$(p,q)$ for $p,q \in \{0,1,2,3,4\}$. 
We measure the performance by taking the lastest $30$ and $60$ predictions at one step ($10 \%$ and $20\%$
aproximately of the data 
set), and computing the quality of the predictions from the mean absolute error of prediction ($MAE$) for last $m$
observations and their respective predictions, that is, \[MAE=\frac{1}{m}%
\sum_{i=1}^{m}\left\vert X_{n-m+i}-\widehat{X}_{n-m+i}\right\vert  \] 
 where $\overline{X}(m):=%
\frac{1}{m}\sum_{i=1}^{m}X_{n-m+i},$ and $X_{1},X_{2},...,X_{n}$ 
 are the real observations, while $%
\widehat{X}_{i}$ are the predictions given by the model for the value $%
X_{i}. $ 
Using this criterion, we obtain that the ARMA$(1,1)$ has the better results under the ARMA$(p,q)$ models
being $MAE=0.8683$ and $MAE=0.6740$ for $m=30$ and $m=60$ predictions respectively. Also, the classical
techniques to validation the model, resulting in a well adjusted by the ARMA$(1,1)$ model. We will show
in this section that we can clearly improve the adjusted ARMA$(1,1)$ model by taking a 
FOU$(\lambda^{(4)},\sigma,H)$.

 According with \cite{kalemkefou}, 
 to fit a FOU$(p)$ model, previously it is necessary to select a filter $a_k$ and a suitable value of $T$.
In this data set, given a filter $a_k$, the value of $MAE$ for different FOU$(p)$ models and 
different values of $T$ are similar. 
Neverthless, the performance were different
in function of the filter considered. In
Figure \ref{MAE_T_filters}, we show that using $T=30$, the minimum 
$MAE$ for $m=30$ predictions was reached for the $a_{26}$ filter. 
According with the theoretical results, see (\cite{kalemkefou}), we need to use $T$ large
but $T/n$ small, for this reason we report the results for $T=30$.
Anyway, under other values of
$T$, the results are similar.
 About the selection of $T$, althought Figure \ref{MAE_T_filters}
suggest to take values for $T \leq 10$, we have selected $T=10$ (to avoid the possibility 
of take $T$ small).
In Table \ref{predictionsfou}, we report the results of $MAE$ for $30$ and $60$ predictions
for the different FOU$(p)$ models considered, using the filter
$a_{26}$ and $T=10$.   Table \ref{predictionsfou}, show that FOU$(\lambda^{(p)},H)$
and FOU$(\lambda^{(p)},\sigma,H)$ (for any $p$)  performs similarly  
 (in several cases the difference is until the fivest 
 decimal),   slightly  better
for FOU$(\lambda^{(p)},H)$ than FOU$(\lambda^{(p)},\sigma,H)$
and
clearly outperforms the family of ARMA$(p,q)$ models. For other side, 
 Figure \ref{acfFOU} show that the observed
 autocorrelation function, is adjusted bad for ARMA$(1,1)$ and FOU$(\lambda^{(p)},H)$ and well
  for FOU$(\lambda^{(p)},\sigma,H)$ being the cases $p=3$ and $p=4$ the best models.

\begin{figure}[H]
 \centering
    \includegraphics[scale=0.41]{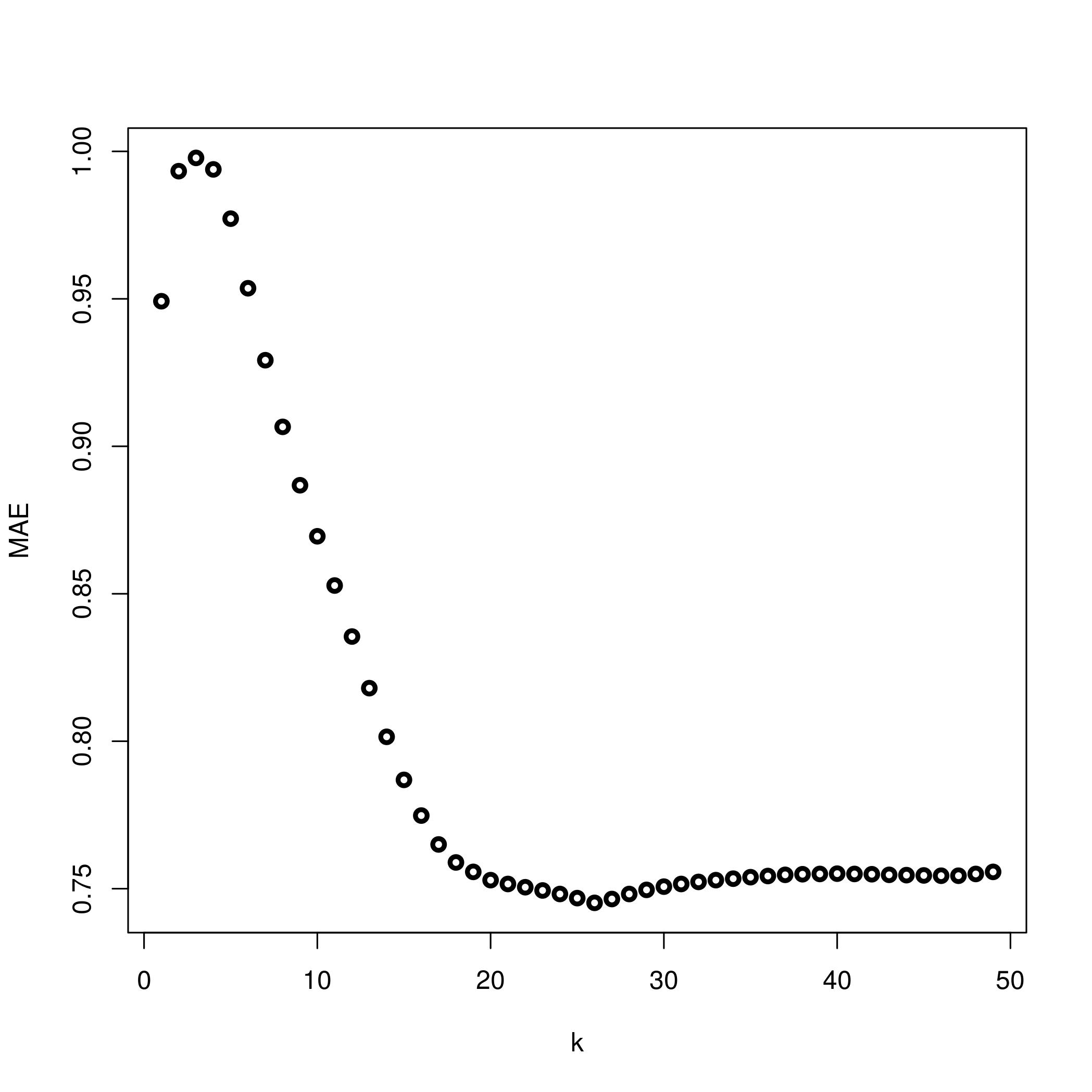}
   \includegraphics[scale=0.41]{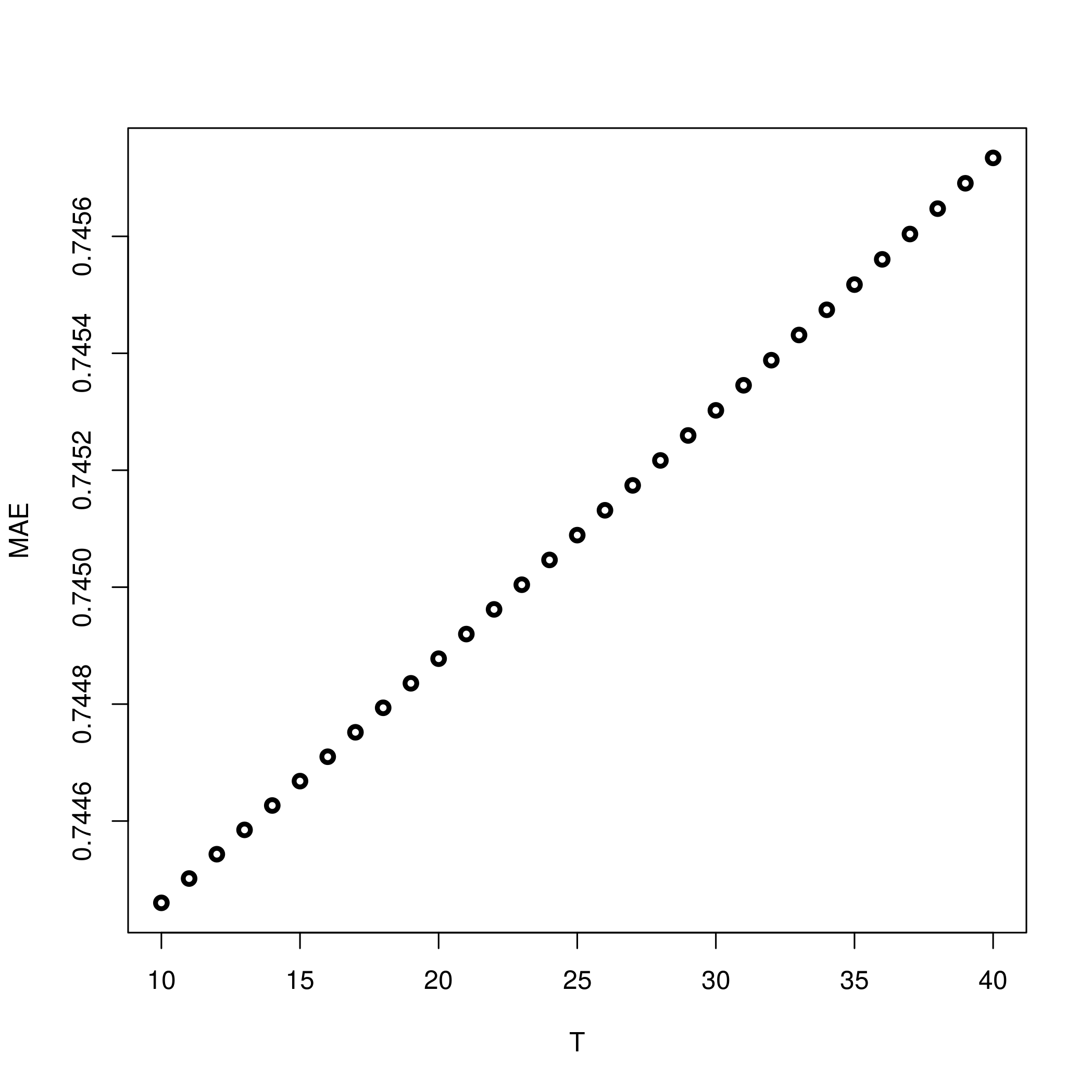} 
   \caption{In the left panel, MAE for $m=30$ predictions for different filters $a_k$ when the data are
   adjusted by FOU$(\lambda^{(2)},H)$ when $T=30$. 
   The minimum is reached at $k=26$. In the right panel,
   MAE for different values of $T$ when the data are
   adjusted by FOU$(\lambda^{(2)},H)$ using the filter $a_{26}$.} 
  \label{MAE_T_filters}
\end{figure}

\begin{figure}[H]
 \centering
    \includegraphics[scale=0.41]{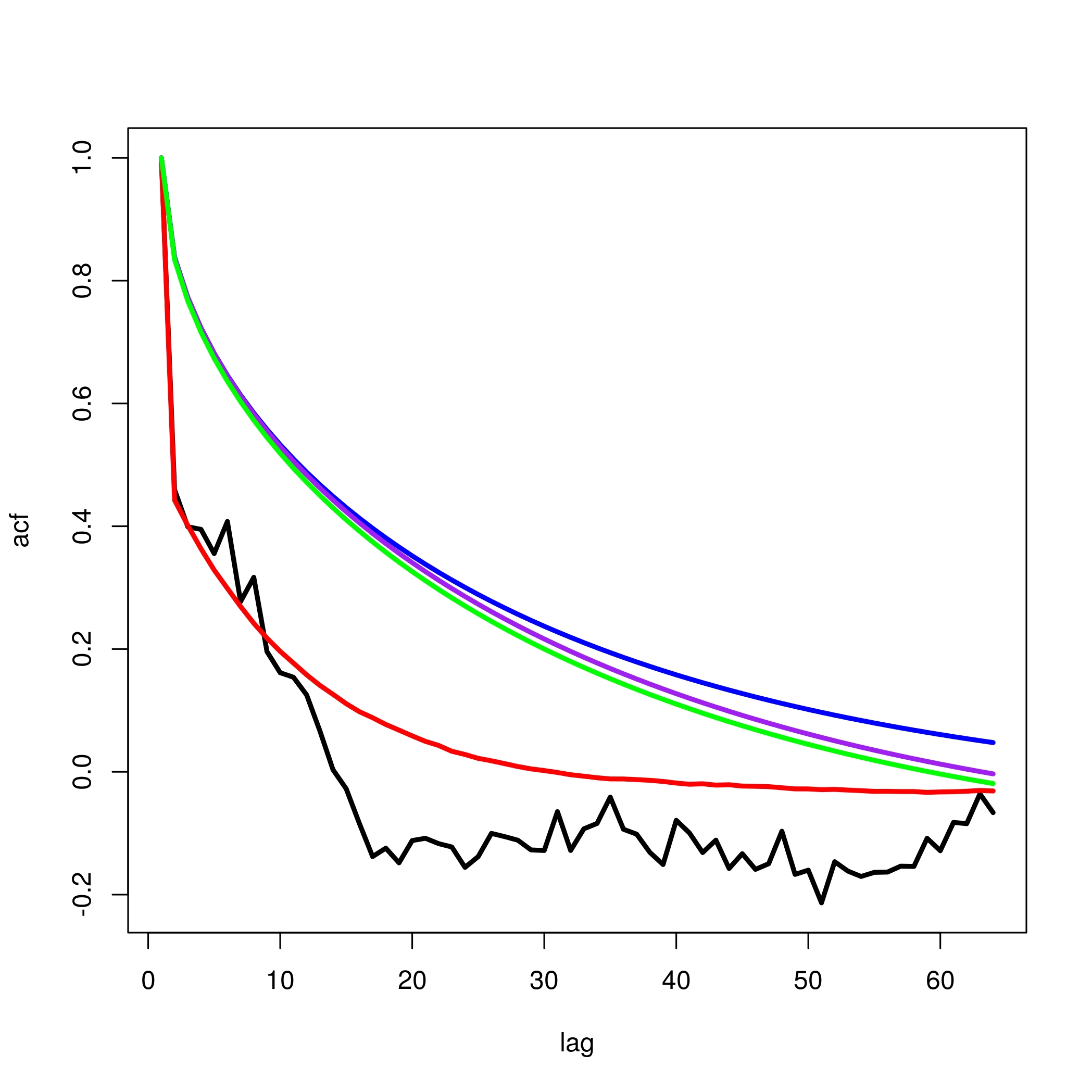}
   \includegraphics[scale=0.41]{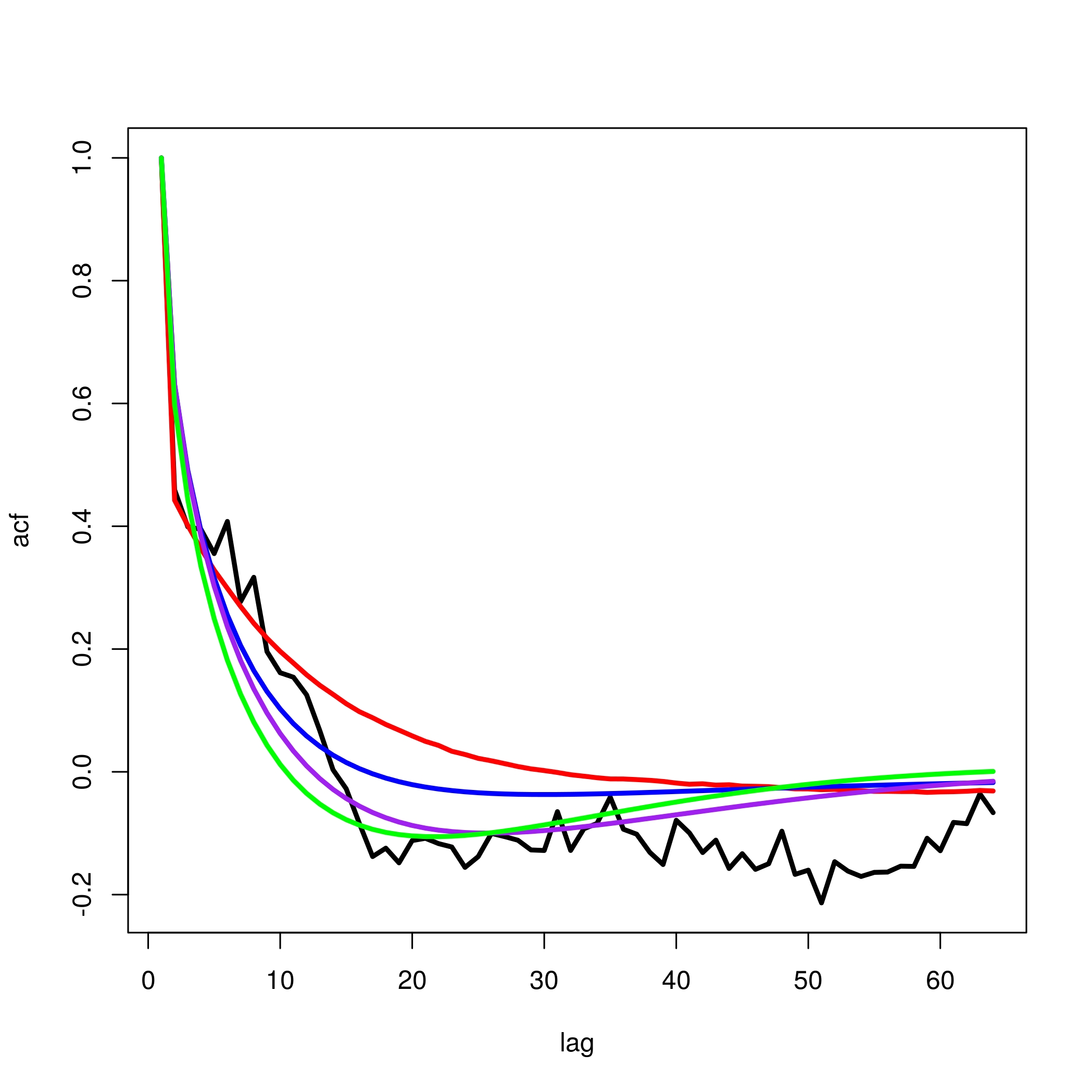} 
   \caption{Autocorrelation function of observed data (black), adjusted ARMA$(1,1)$ and adjusted FOU$(\lambda^{(p)},H)$
   in the left panel and FOU$(\lambda^{(p)},\sigma,H)$ in the right panel for $p=1$ (blue), $p=3$ (green) and
   $p=4$ (purple).}
      \label{acfFOU}
\end{figure}

 \begin{table*}[ht]\caption{$MAE$ for $30$ and $60$ predictions ($MAE_{30}$ and $MAE_{60}$
 respectively) for different  FOU$(p)$ models considered, using
 $T=10$ and $a_{26}$.}
 \label{predictionsfou}

 \centering
\begin{tabular}{|c|ccc|||c|c|}

    \hline
     & $\hat{H}$ & $\hat{\sigma}$ & $\hat{\lambda}$  & $MAE_{30}$ & $MAE_{60}$ \\
   \hline 
      FOU$(\lambda,H)$ &0.2468 &--- &0.1974 &0.7434 &0.6248 \\

      FOU$(\lambda,\sigma,H)$ &0.2468 &2.3044 &1.0710 &0.7509 &1.0700 \\

   FOU$(\lambda^{(2)},H)$ & 0.2468 & --- & 0.1112 & 0.7434 & 0.6249 \\
      FOU$(\lambda^{(2)},\sigma,H)$ &0.2468 & 2.3044  & 0.6032 & 0.7537 & 0.6280 \\
         FOU$(\lambda^{(3)},H)$  &0.2468 &--- & 0.0851 &0.7434 &0.6250 \\
               FOU$(\lambda^{(3)},\sigma,H)$ &0.2468 &2.3044 &0.4619 &0.7555 &0.6280  \\
                     FOU$(\lambda^{(4)},H)$ & 0.2468 & --- & 0.0715 & 0.7434 & 0.6250  \\
                             FOU$(\lambda^{(4)},\sigma,H)$ & 0.2468 &2.3044 &0.3882 &0.7567 &0.6279  \\
 ARMA$(1,1)$ &--- &--- &--- &0.8683 &0.6740 \\

           \hline
     
\end{tabular}
\end{table*}

\begin{remark}
 The values of $\hat{H}$ and $\hat{\sigma}$ showed in Table \ref{predictionsfou}, are 
 the same in all the models considered, because
 the estimation of both of them are independent of $p$.
\end{remark}
\begin{remark}\label{remark_FOUlambda_p}
 To model a time series from a FOU$(p)$ process, we have an apparent disvantage to take 
 only one $\lambda$ than several. Neverthless in the three real data set worked in 
 \cite{kalemkefou} we have observed no substantial difference between the performance of 
 FOU$(\lambda_1,\lambda_2,\sigma,H)$ or  FOU$(\lambda_1,\lambda_2,\lambda_3, \sigma,H)$ than 
  FOU$(\lambda^{(2)},\sigma,H)$ or FOU$(\lambda^{(3)},\sigma,H)$ and in the application 
  of this work either. 
\end{remark}

To conclude this section, we have observed that the FOU$(\lambda^{(3)},\sigma, H)$ and
FOU$(\lambda^{(4)},\sigma, H)$ models outperforms clearly the family of the ARMA$(p,q)$ models.
\section{Conclusions}\label{conclusions}
According with Remark \ref{remark_FOUlambda_p}, there is no substantial loss when we model a time series using
FOU$(\lambda^{(p)},\sigma,H)$ processes instead the more general FOU$(p)$ processes.
In this work we have proposed a new method to estimate $\lambda$ in a FOU$(\lambda^{(p)},\sigma,H)$
process. We showed that this
new method has several advantages. For the one hand, this new method is more easy and fast 
to compute because it is provenient by an explicit formula. On the other hand, it only requires to 
have observed the process in a equispaced sample of $[0,T]$, and we have proved consistency and asymptotic
normality (at least for $1/2<H<3/4$). 
In this way, we can estimate the three parameters of the model using explicit formulas, avoiding the possible approximation errors
of the numerical approximations and estimating more efficiently.
By simulations, we show that 
the new method to estimate $\lambda$ work well and is more efficient
than the proposed in \cite{kalemkefou}. Lastly, we include an application to real data, 
and we show that the new method work well too and outperforms the familiy of ARMA$(p,q)$.
To finish, we can say that the FOU$(p)$ processes can be considered as an alternative to ARMA (or ARFIMA) processes
to model time series and in this work, we give a way to estimate their parameters efficiently and with desirable
asymptotic properties.

\section{Proofs}\label{proofs}

To prove Proposition \ref{fouvariance} we need show the following two lemmas.
\begin{lemma}\[\]
Let $p \geq 2.$
The function $g(H)$ defined as
 \begin{equation*}
g(H)=\frac{\left( 2H-1\right) }{\Gamma \left( 2H\right) }\sum_{i,j=0}^{p-1}%
\frac{\binom{p-1}{i}\binom{p-1}{j}\left( -1\right) ^{i+j}}{i!j!}\mathbb{E}%
\left( \int_{0}^{+\infty }u^{i}e^{-u}du\int_{0}^{+\infty
}v^{j}e^{-v}\left\vert u-v\right\vert ^{2H-2}dv\right)
\end{equation*}

is a polynomial of degree $p-1$ with zeros in $1,2,...,p-1.$
\label{poly1}
\end{lemma}

\begin{lemma}\[\]
Let $p \geq 2$. Then, the function $g(H)$ defined in Lemma \ref{poly1} is
 \begin{equation*}
g(H)=\frac{\prod_{i=1}^{p-1}\left( i-H\right) }{\left( p-1\right) !}.
\end{equation*}
\label{poly2}
\end{lemma}
\begin{proof}[Proof of Lemma \ref{poly1}.] \[\]
For every $i,j=0,1,2,...,p-1,$ define 
\begin{equation}\label{gij}
 g_{ij}\left( H\right) =\mathbb{E}%
\left( \int_{0}^{+\infty }u^{i}e^{-u}du\int_{0}^{+\infty
}v^{j}e^{-v}\left\vert u-v\right\vert ^{2H-2}dv\right) .
\end{equation}

Then 
\begin{equation*}
g(H)=\frac{\left( 2H-1\right) }{\Gamma \left( 2H\right) }\sum_{i,j=0}^{p-1}%
\frac{\binom{p-1}{i}\binom{p-1}{j}\left( -1\right) ^{i+j}}{i!j!}g_{ij}\left(
H\right) .
\end{equation*}

If we make $x=u+v$, $v=u-v$, we obtain that $g_{ij}\left( H\right) =$

\begin{equation*}
\frac{1}{2^{i+j+1}}\int_{0}^{+\infty }e^{-x}dx\int_{0}^{x}\left( x+y\right)
^{i}\left( x-y\right) ^{j}y^{2H-2}dy+
\end{equation*}

\begin{equation*}
 \frac{1}{2^{i+j+1}}\int_{0}^{+\infty
}e^{-x}dx\int_{-x}^{0}\left( x+y\right) ^{i}\left( x-y\right) ^{j}\left(
-y\right) ^{2H-2}dy=
\end{equation*}%
\begin{equation*}
A_{ij}(H)+B_{ij}(H).
\end{equation*}

\begin{equation*}
A_{ij}(H)=\frac{1}{2^{i+j+1}}\int_{0}^{+\infty
}e^{-x}dx\int_{0}^{x}\sum_{h=0}^{i}\binom{i}{h}x^{h}y^{i-h}\sum_{k=0}^{j}%
\binom{j}{k}x^{k}\left( -y\right) ^{j-k}y^{2H-2}dy=
\end{equation*}%
\begin{equation*}
\frac{1}{2^{i+j+1}}\sum_{h=0}^{i}\binom{i}{h}\sum_{k=0}^{j}\binom{j}{k}%
\left( -1\right) ^{j-k}\int_{0}^{+\infty
}x^{h+k}e^{-x}dx\int_{0}^{x}y^{i+j-h-k+2H-2}dy=
\end{equation*}%
\begin{equation*}
\frac{\Gamma \left( i+j+2H\right) \left( -1\right) ^{j}}{2^{i+j+1}}%
\sum_{h=0}^{i}\binom{i}{h}\sum_{k=0}^{j}\frac{\binom{j}{k}\left( -1\right)
^{k}}{i+j-h-k+2H-1}.
\end{equation*}

\begin{equation*}
B_{ij}(H)=\frac{1}{2^{i+j+1}}\int_{0}^{+\infty }e^{-x}dx\int_{-x}^{0}\left(
x+y\right) ^{i}\left( x-y\right) ^{j}\left( -y\right) ^{2H-2}dy=
\end{equation*}%
\begin{equation*}
\frac{1}{2^{i+j+1}}\int_{0}^{+\infty }e^{-x}dx\int_{0}^{x}\left( x-y\right)
^{i}\left( x+y\right) ^{j}y^{2H-2}dy=A_{ji}(H).
\end{equation*}

\bigskip

Then%
\begin{equation*}
g\left( H\right) =\frac{\left( 2H-1\right) }{\Gamma \left( 2H\right) }%
\sum_{i,j=0}^{p-1}\frac{\binom{p-1}{i}\binom{p-1}{j}\left( -1\right) ^{i+j}}{%
i!j!}\left( A_{ij}(H)+A_{ji}(H)\right) =
\end{equation*}%
\begin{equation}
\frac{\left( 2H-1\right) }{\Gamma \left( 2H\right) }\sum_{i,j=0}^{p-1}\frac{%
\binom{p-1}{i}\binom{p-1}{j}\left( -1\right) ^{i}\Gamma \left( i+j+2H\right) 
}{i!j!2^{i+j}}\sum_{h=0}^{i}\sum_{k=0}^{j}\frac{\binom{i}{h}\binom{j}{k}%
\left( -1\right) ^{k}}{i+j-h-k+2H-1}.
\label{53}
\end{equation}

If we replacing in the last equality the expression $\Gamma \left(
i+j+2H\right) $ for

$\left( i+j+2H-1\right) \left( i+j+2H-2\right) ...\left( 1+2H\right)
2H\Gamma \left( 2H\right) $ we obtain that 
\begin{equation*}
g(H)=\left( 2H-1\right)\times 
\end{equation*}

\begin{equation}
 \sum_{i,j=0}^{p-1}\frac{\binom{p-1}{i}\binom{p-1}{j}%
\left( -1\right) ^{i}\left( i+j+2H-1\right) ...\left( 1+2H\right) 2H}{%
i!j!2^{i+j}}\sum_{h=0}^{i}\sum_{k=0}^{j}\frac{\binom{i}{h}\binom{j}{k}\left(
-1\right) ^{k}}{i+j-h-k+2H-1}
\label{98}
\end{equation}%
where in the case $i=j=0,$ the expression

$\left( i+j+2H-1\right) \left( i+j+2H-2\right) ...\left( 1+2H\right)
2H\Gamma \left( 2H\right) $ it means $\Gamma \left( 2H\right) .$

Observing that in the case $i=j=0$ we have $2H$ and in the rest of summands
(where $i+j\geq 1$) we have powers of $H$ (because \ for any $h,k$ the
expression $i+j-h-k+2H-1$ appears in the expanssion $\left( i+j+2H-1\right)
...\left( 1+2H\right) H$). This concludes the proof that $g$ is a polynomial.

To prove that $g$ has degree $p-1$, observe that in the case $H>1/2$ we can
write 
\begin{equation*}
\sum_{h=0}^{i}\sum_{k=0}^{j}\frac{\binom{i}{h}\binom{j}{k}\left( -1\right)
^{k}}{i+j-h-k+2H-1}=\sum_{h=0}^{i}\sum_{k=0}^{j}\frac{\binom{i}{h}\binom{j}{k%
}\left( -1\right) ^{j-k}}{h+k+2H-1}=
\end{equation*}%
\begin{equation*}
\sum_{h=0}^{i}\binom{i}{h}\sum_{k=0}^{j}\binom{j}{k}\left( -1\right)
^{j-k}\int_{0}^{1}x^{h+k+2H-2}dx=
\end{equation*}%
\begin{equation*}
\sum_{h=0}^{i}\binom{i}{h}\int_{0}^{1}x^{h+2H-2}\sum_{k=0}^{j}\binom{j}{k}%
\left( -1\right) ^{j-k}x^{k}dx=\left( -1\right) ^{j}\sum_{h=0}^{i}\binom{i}{h%
}\int_{0}^{1}x^{h+2H-2}\left( 1-x\right) ^{j}dx=
\end{equation*}%
\begin{equation}
\left( -1\right) ^{j}\sum_{h=0}^{i}\binom{i}{h}\frac{\Gamma \left(
h+2H-1\right) j!}{\Gamma \left( h+j+2H\right) }.
\label{101}
\end{equation}%
Putting (\ref{101}) in (\ref{98}) we obtain that 
\begin{equation*}
g(H)=\left( 2H-1\right)\times
\end{equation*}

\begin{equation*}
 \sum_{i,j=0}^{p-1}\frac{\binom{p-1}{i}\binom{p-1}{j}%
\left( -1\right) ^{i+j}\left( i+j+2H-1\right) ...\left( 1+2H\right) 2H}{%
i!2^{i+j}}\sum_{h=0}^{i}\binom{i}{h}\frac{\Gamma \left( h+2H-1\right) }{%
\Gamma \left( h+j+2H\right) }=
\end{equation*}%
\begin{equation*}
\left( 2H-1\right)\times 
\end{equation*}

\begin{equation*}
 \sum_{i,j=0}^{p-1}\sum_{h=0}^{i}\frac{\binom{p-1}{i}\binom{p-1}{j}\binom{i}{h}%
\left( -1\right) ^{i+j}\left( i+j+2H-1\right) ...\left( 1+2H\right) 2H}{%
i!2^{i+j}\left( h+j+2H-1\right) ...\left(
h+2H-1\right)}.
\end{equation*}

Then, $g$ is a polynomial of $p-1$ degree.

Observe that in the case in which $H<1/2$, the integral $%
\int_{0}^{1}x^{h+k+2H-2}dx$ does not exist when $h=k=0$, but the results
ramians valid if we separate the case $h=k=0$ and the case $h+k\geq 1.$

To prove that $g(1)=g(2)=...=g(p-1)=0$, for values of $H=1,2,...,p-1$ we can
develop the binomial formula for $\left\vert u-v\right\vert ^{2H-2}$ and we
obtain that 
\begin{equation*}
\mathbb{E}\left( \int_{0}^{+\infty }u^{i}e^{-u}du\int_{0}^{+\infty
}v^{j}e^{-v}\left\vert u-v\right\vert ^{2H-2}dv\right) =
\end{equation*}

\begin{equation*}
\mathbb{E}\left(
\int_{0}^{+\infty }u^{i}e^{-u}du\int_{0}^{+\infty
}v^{j}e^{-v}\sum_{k=0}^{2H-2}\binom{2H-2}{k}u^{k}v^{2H-2-k}dv\right) =
\end{equation*}%
\begin{equation*}
\sum_{k=0}^{2H-2}\binom{2H-2}{k}\mathbb{E}\left( \int_{0}^{+\infty
}u^{i+k}e^{-u}du\int_{0}^{+\infty }v^{j+2H-2-k}e^{-v}dv\right)
=
\end{equation*}

\begin{equation}
 \sum_{k=0}^{2H-2}\binom{2H-2}{k}\left( i+k\right) !\left( j+2H-2-k\right) !.
 \label{205}
\end{equation}

\bigskip Putting (\ref{205}) in  (\ref{gij}) we
obtain that 
\begin{equation*}
g(H)=\frac{\left( 2H-1\right) }{\Gamma \left( 2H\right) }\sum_{i,j=0}^{p-1}%
\frac{\binom{p-1}{i}\binom{p-1}{j}\left( -1\right) ^{i+j}}{i!j!}%
\sum_{k=0}^{2H-2}\binom{2H-2}{k}\left( i+k\right) !\left( j+2H-2-k\right) !=
\end{equation*}%
\begin{equation*}
\frac{\left( 2H-1\right) }{\Gamma \left( 2H\right) }\sum_{k=0}^{2H-2}\binom{%
2H-2}{k}\sum_{i=0}^{p-1}\frac{\binom{p-1}{i}\left( i+k\right) !\left(
-1\right) ^{i}}{i!}\sum_{i=0}^{p-1}\frac{\binom{p-1}{j}\left( -1\right)
^{j}\left( j+2H-2-k\right) !}{j!}.
\end{equation*}%
Therefore, It is enough to show that 
\begin{equation*}
\sum_{i=0}^{p-1}\frac{\binom{p-1}{i}\left( i+k\right) !\left( -1\right) ^{i}%
}{i!}=0\text{ for every }H=1,2,...,p-1\text{ and }k=0,1,2,...,2H-2.
\end{equation*}

This result it follows from the binomial formula of $\alpha (x)=\left(
1-x\right) ^{p-1}$ and using that $\alpha \left( 1\right) =\alpha ^{\prime
}\left( 1\right) =...=\alpha ^{\left( p-1\right) }\left( 1\right) =0.$

This concludes the proof that $g(1)=g(2)=...=g(p-1)=0.$
 
\end{proof}
\begin{proof}[Proof of Lemma \ref{poly2}.]\[\]
 
From (\ref{53}) we deduce that $g(0)=1$, therefore the corollary follows immediately
from Lemma \ref{poly1}.
\end{proof}
\begin{proof}[Proof of Proposition \ref{fouvariance}.]\[\]
It is enough to consider the case $p \geq 2$, because when $p=1$ we have that (\ref{V(Xt)})
is the well known variance of a fractional Ornstein-Uhlenbeck process.
 If $\left\{ X_{t}\right\} _{t\in \mathbb{R}}\sim $FOU$\left( \lambda
^{\left( p\right) },\sigma ,H\right) $ then $X_{t}=\sigma \sum_{i=0}^{p-1}%
\binom{p-1}{i}T_{\lambda }^{\left( i\right) }\left( B_{H}\right) (t)$ where $%
\left\{ B_{H}\left( t\right) \right\} _{t\in \mathbb{R}}$ is a fractional
Brownian motion with Hurst parameter $H$ and the operators $T_{\lambda
}^{\left( i\right) }$ are defined in (\ref{hh}), thus, it is enough to prove the
formula in the case in which $\sigma =1.$ Therefore, if $\left\{
X_{t}\right\} _{t\in \mathbb{R}}\sim $FOU$\left( \lambda ^{\left( p\right)
},1,H\right) $, then 
\begin{equation*}
\mathbb{V}\left( X_{t}\right) =\mathbb{E}\left( X_{0}^{2}\right) =\mathbb{E}%
\left( \sum_{i,j=0}^{p-1}\binom{p-1}{i}\binom{p-1}{j}T_{\lambda }^{\left(
i\right) }\left( B_{H}\right) (0)T_{\lambda }^{\left( j\right) }\left(
B_{H}\right) (0)\right) =
\end{equation*}%
\begin{equation*}
\mathbb{E}\left( \sum_{i,j=0}^{p-1}\binom{p-1}{i}\binom{p-1}{j}\int_{-\infty
}^{0}\frac{\left( \lambda w\right) ^{i}}{i!}e^{\lambda w}dB_{H}\left(
w\right) \int_{-\infty }^{0}\frac{\left( \lambda z\right) ^{j}}{j!}%
e^{\lambda z}dB_{H}\left( z\right) \right) =
\end{equation*}%
\begin{equation}
\mathbb{E}\left( H\left( 2H-1\right) \sum_{i,j=0}^{p-1}\binom{p-1}{i}\binom{%
p-1}{j}\int_{-\infty }^{0}\frac{\left( \lambda w\right) ^{i}}{i!}e^{\lambda
w}dw\int_{-\infty }^{0}\frac{\left( \lambda z\right) ^{j}}{j!}e^{\lambda
z}\left\vert w-z\right\vert ^{2H-2}dz\right) .
\label{253}
\end{equation}%

The last equality in \ref{253} is due to 
the following formula, whose proof can be seen in  \cite{Pipiras}: 
if $H\in \left( 1/2,1\right) $ and \[f,g\in \left\{ f:\mathbb{R\rightarrow R}\text{: }\int \int_{\mathbb{R}%
^{2}}\left\vert f(u)f(v)\right\vert \left\vert u-v\right\vert
^{2H-2}dudv<+\infty \right\}, \] then 
\begin{equation}
\mathbb{E}\left( \int_{-\infty }^{+\infty }f(u)dB_{H}(u)\int_{-\infty
}^{+\infty }g(v)dB_{H}(v)\right) =  \label{pipiras}
\end{equation}%
\begin{equation*}
H(2H-1)\int_{-\infty }^{+\infty }f(u)du\int_{-\infty }^{+\infty
}g(v)\left\vert u-v\right\vert ^{2H-2}dv.
\end{equation*}

If we change $\lambda w=-u$ and $\lambda z=-v$ we obtain that (\ref{253}) is equal to 
\begin{equation*}
\frac{H\left( 2H-1\right) }{\lambda ^{2H}}\sum_{i,j=0}^{p-1}\frac{\binom{p-1%
}{i}\binom{p-1}{j}\left( -1\right) ^{i+j}}{i!j!}\mathbb{E}\left(
\int_{0}^{+\infty }u^{i}e^{-u}du\int_{0}^{+\infty }v^{j}e^{-v}\left\vert
u-v\right\vert ^{2H-2}dv\right) =
\end{equation*}%
\begin{equation*}
\frac{H\Gamma \left( 2H\right) }{\lambda ^{2H}}g(H)
\end{equation*}%
where $g(H)$ is the function defined in Lemma \ref{poly1}. 
From Lemma \ref{poly2} we obtain the result.
This concludes the proof.
\end{proof}
\begin{proof}[Proof of Theorem \ref{consistency}.]\[\]
Throughout this theorem we will call $\lambda ^{0},\sigma ^{0}$ and $H^{0}$
the true value of the parameters, also we will call $\mu _{2}^{0}$ the true
value of the $\mathbb{V}\left( X_{t}\right) $ given in (\ref{V(Xt)}).

Observe that $\widehat{\lambda }=G\left( \widehat{\sigma },\widehat{H},%
\widehat{\mu }_{2}\right) $ where $G\left( \sigma ,H,\mu _{2}\right) =\left( 
\frac{\sigma ^{2}H\Gamma \left( 2H\right) \prod_{i=1}^{p-1}\left(
i-H\right) }{\left( p-1\right) !\mu _{2}}\right) .$

From the ergodic theorem we know that $\frac{1}{T}\int_{0}^{T}X_{t}^{2}dt%
\overset{a.s.}{\rightarrow }\mu _{2}^{0}$.

Any FOU$\left( \lambda ^{\left( p\right) },\sigma ,H\right) $ is a Gaussian
process with H\"{o}lder index $H$, then, the conditions $1/2<H<3/4$ and $%
n\left( \frac{T}{n}\right) ^{k}\rightarrow 0$ as $n\rightarrow +\infty $ for
some $k>1$ allows to affirm that $\sqrt{T}\left( \frac{1}{T}%
\int_{0}^{T}X_{t}^{2}dt-\widehat{\mu }_{2}\right) \overset{P}{\rightarrow }0$
(Lemma 8 in (\cite{Kessler})), thus $\widehat{\mu }_{2}\overset{P}{\rightarrow }%
\mu _{2}^{0}.$

\bigskip From continuity of $G$ we obtain immediately that 
\begin{equation*}
\widehat{\lambda }=G\left( \widehat{\sigma },\widehat{H},\widehat{\mu }%
_{2}\right) \overset{a.s.}{\rightarrow }G\left( \sigma ^{0},H^{0},\mu
_{2}^{0}\right) =\lambda ^{0}.
\end{equation*}

Applying the mean value theorem we have 
\begin{equation*}
G\left( \widehat{\sigma },\widehat{H},\widehat{\mu }_{2}\right) -G\left(
\sigma ^{0},H^{0},\mu _{2}^{0}\right) =\nabla G\left( \widetilde{\sigma },%
\widetilde{H},\widetilde{\mu }_{2}\right) .\left( \widehat{\sigma }-\sigma
^{0},\widehat{H}-H,\widehat{\mu }_{2}-\mu _{2}^{0}\right) 
\end{equation*}%
where $\left( \widetilde{\sigma },\widetilde{H},\widetilde{\mu }_{2}\right)
\in \left[ \left( \widehat{\sigma },\widehat{H},\widehat{\mu }_{2}\right)
,\left( \sigma ^{0},H^{0},\mu _{2}^{0}\right) \right] .$ Then 
\begin{equation*}
 \sqrt{T}\left ( \hat{\lambda}-\lambda^{0}\right )=\sqrt{T}%
\left( G\left( \widehat{\sigma },\widehat{H},\widehat{\mu }_{2}\right)
-G\left( \sigma ^{0},H^{0},\mu _{2}^{0}\right) \right) =
\end{equation*}
 \ 
\begin{equation*}
\sqrt{T}\left( \frac{\partial G\left( \widetilde{\sigma },\widetilde{H},%
\widetilde{\mu }_{2}\right) }{\partial \sigma }\left( \widehat{\sigma }%
-\sigma ^{0}\right) +\frac{\partial G\left( \widetilde{\sigma },\widetilde{H}%
,\widetilde{\mu }_{2}\right) }{\partial H}\left( \widehat{H}-H\right) +\frac{%
\partial G\left( \widetilde{\sigma },\widetilde{H},\widetilde{\mu }%
_{2}\right) }{\partial \mu _{2}}\left( \widehat{\mu }_{2}-\mu
_{2}^{0}\right) \right) .
\end{equation*}

Observe that the derivatives of $G$ with respect to $\sigma ,H$ and $\mu _{2}
$ are bounded in a neigbourhood of $\left( \sigma ^{0},H^{0},\mu
_{2}^{0}\right) .$

From  Theorem \ref{asymptotic of H sigma} and condition $\frac{T\log ^{2}n}{n}%
\rightarrow 0$ as $n\rightarrow +\infty ,$ \ we have that  \\ $\sqrt{T}\frac{%
\partial G\left( \widetilde{\sigma },\widetilde{H},\widetilde{\mu }%
_{2}\right) }{\partial \sigma }\left( \widehat{\sigma }-\sigma ^{0}\right) 
\overset{P}{\rightarrow }0$ and $\sqrt{T}\frac{\partial G\left( \widetilde{%
\sigma },\widetilde{H},\widetilde{\mu }_{2}\right) }{\partial H}\left( 
\widehat{H}-H\right) \overset{P}{\rightarrow }0.$
Therefore, the asymptotic distribution of 
$\sqrt{T}\left ( \hat{\lambda}-\lambda^{0}\right )$ is the same as that of

\begin{equation*}
\sqrt{T}\frac{\partial G\left( \widetilde{\sigma },\widetilde{H},\widetilde{\mu }%
_{2}\right) }{\partial \mu _{2}}\left( \widehat{\mu }_{2}-\mu
_{2}^{0}\right) =
\end{equation*}

\begin{equation*}
 \sqrt{T}\frac{\partial G\left( \widetilde{\sigma },\widetilde{H},%
\widetilde{\mu }_{2}\right) }{\partial \mu _{2}}\left( \widehat{\mu }_{2}-%
\frac{1}{T}\int_{0}^{T}X_{t}^{2}dt\right) +\sqrt{T}\frac{\partial G\left( \widetilde{%
\sigma },\widetilde{H},\widetilde{\mu }_{2}\right) }{\partial \mu _{2}}%
\left( \frac{1}{T}\int_{0}^{T}X_{t}^{2}dt-\mu _{2}^{0}\right) .
\end{equation*}
 
\end{proof}


\end{document}